\documentclass[12pt]{amsart}

\usepackage{amscd}
\usepackage{amssymb,amsmath,amsthm}
\theoremstyle{plain}
\newtheorem{thm}{Theorem}[section]
\newtheorem{lem}[thm]{Lemma}
\newtheorem{cor}[thm]{Corollary}
\newtheorem{defn-lem}[thm]{Definition-Lemma}

\newtheorem{prop}[thm]{Proposition}

\theoremstyle{definition}
\newtheorem{defn}[thm]{Definition}

\newtheorem{rem}[thm]{Remark}

\def\md #1#2#3#4#5 {\left(
                        \begin{matrix}
             #1 & #2 \\
             #3 & #4
                        \end{matrix}
                      \right)- #5}

\def\ma #1#2#3#4 {\left(
                        \begin{matrix}
             #1 & #2 \\
             #3 & #4
                        \end{matrix}
                      \right)}
\def\mu (#1) {\mathcal{M}#1}

\def \mb {\mathbb}

\def \Ker {\operatorname{Ker}}
\def \Im {\operatorname{Im}}
\def\Ind{\operatorname{Ind}}

\def\rank{\operatorname{rank}}

\def\KK{\operatorname{KK}}
\def\Hom{\operatorname{Hom}}
\newcommand{\mc}{\mathcal}

\newcommand{\mf}{\mathfrak}
\begin{document}
\title [Deformation of projections and projection lifting ]
       {Deformation of a projection in the multiplier algebra and projection lifting from the corona algebra }

\begin{abstract}
Let $X$ be a unit interval or a unit circle and let $B$ be a $\sigma_p$-unital, purely infinite, simple $C\sp*$-algebra such that its multiplier algebra $M(B)$ has real rank zero. Then we determine necessary and sufficient conditions for a projection in the corona algebra of $C(X)\otimes B$ to be liftable to a projection in the multiplier algebra. This generalizes a result proved by L. Brown and the author \cite{BL}. The main technical tools are divided into two parts. The first part is borrowed from the author's result, \cite[Theorem 3.3]{Lee}. The second part is a proposition showing that we can produce a sub-projection, with an arbitrary rank which is prescribed as K-theoretical data, of a projection or a co-projection in the multiplier algebra of $C(X)\otimes B$ under a suitable ``infinite rank and co-rank'' condition.
\end{abstract}

\author { Hyun Ho \quad Lee }

\address {Department of Mathematics\\
         University of Ulsan\\
         Ulsan, South Korea 680-749 }
\email{hadamard@ulsan.ac.kr}


\date{June 30 2012}
\thanks{Research partially supported by }
\maketitle

\section{Introduction}
Let $X$ be a (finite dimensional) locally compact Hausdorff space and $B$ a $C\sp*$-aglebra. We are concerned with the projection lifting problem from the corona algebra  to the multiplier algebra of a $C\sp*$-algebra of the form $C(X)\otimes B$. An earlier result in this direction is that of W. Calkin \cite{ca}, who showed that a projection in the quotient algebra $B(H)/K$ is liftable to a projection in B(H) where $H$ is a separable infinite dimensional Hilbert space and $K$ is the $C\sp*$-algebra of compact operators on $H$; in our setting this corresponds to the case that $X$ is a one point set and $B$ is the $C\sp*$-algebra of compact operators.
 A generalization of this result was given by L. Brown and the author as follows.
 \begin{thm}\label{T:projectionlifting}
  Let $X$ be $[0,1]$, $(-\infty,\infty)$, $[0,\infty)$, or a unit circle. A projection $\mathbf{f}$ in the corona algebra of $C(X)\otimes K $, represented by finite umber of projection valued functions $(f_0,\cdots, f_n)$ under a suitable partition $\{ x_1,\dots, x_n \}$ of the interior of $X$, is liftable to a projection in the multiplier algebra if and only if there
  exist $l_0,\cdots,l_n$ satisfying the following conditions; suppose $k_i$'s are the essential codimensions of $f_i(x_i)$ and $f_{i-1}(x_i)$ for  $1 \leq i \leq n$.
  \begin{equation}\label{E:eq1}
 l_i-l_{i-1}=-k_i \quad \mbox{for} \quad i>0 \quad \mbox{and}\quad l_0-l_n=-k_0 \quad \mbox{in the circle case,}
 \end{equation}  if for some $x$ in $X_i$, $f_i(x)$ has finite rank, then
             \begin{equation}\label{E:eq2}
         l_i \geq - \rank(f_i(x)),
         \end{equation} if  for some $x$ in $X_i$, $1-f_i(x)$ has finite rank, then
  \begin{equation}\label{E:eq3}
  l_i \leq \rank(1-f_i(x)),
  \end{equation} if either end point of $X_i$ is infinite, then
  \begin{equation}\label{E:eq4}
  l_i=0.
  \end{equation}
  \end{thm}

 We note that our result was obtained from the following interesting proposition on continuous fields of Hilbert spaces. This result says, roughly speaking, that a continuous filed of Hilbert spaces such that each fiber's rank is greater and equal to $n$ has a trivial subfield of rank $m$ for any $m \leq n$. (See also \cite[Proposition 3.2]{DNNP}.)

\begin{prop}(\cite[Corollary A.5]{BL})\label{L:subprojection}
If $X$ is a separable metric space such that whose covering dimension is less than or equal to 1 and
$\mathcal{H}$ is a continuous field of Hilbert spaces over $X$ such
that $\dim (H_x) \geq n $ for every $x \in X$, then $\mathcal{H}$
has a trivial subfield of rank n. Equivalently, if $p$ is a strongly
continuous projection valued function on $X$ such that
$\rank(p(x))\geq n$ for every $x \in X$, then there is a norm
continuous projection valued function $q$ such that $q \leq p$ and
$\rank(q(x))=n$ for every $x \in X$.
\end{prop}

We denote by $M(B)$ the mutilpier algebra of a stable $C\sp*$-algebra $B$, and consider the  projection valued map $\mathbf{p}:X \to M(B)$, which is continuous with respect to the strict topology. Note that the associated fiber $p(x)H_B$ is a Hilbert submodule so that we have no proper notion of rank. But at least we can distinguish ``finiteness'' and ``infiniteness'' of a Hilbert (sub) module using finiteness and infiniteness of the corresponding fiberwise projection in $M(B)$.  We are going to show that:
\begin{lem}\label{L:deform1}
Let $\mathbf{p}$ be a continuous section from $X$ to $M(B)$ with respect to the strict topology on $M(B)$  where  $B$ is a $\sigma_p$-unital, stable $C\sp*$-algebra of real rank zero such that $M(B)$ contains a halving full projection. In addition, when we denote its image on $x\in X$ by $p_x$, assume that $p_x$ is a full, properly infinite projection for each $x$. Then for any $\alpha \in K_0(B)$ there exists a norm continuous section $\mathbf{r}$ from $X$ to $B$ such that  $r_x \leq p_x$ such that $[r]_{K_0(B)}=\alpha$.
\end{lem}

Then using the above lemma we can generalize Theorem \ref{T:projectionlifting} to the case $B$ is a $\sigma_p$-unital, purely infinite simple $C\sp*$-algebra provided that $K_0(B)$ is an ordered group.
 \begin{thm}[Theorem \ref{T:liftingthm}]\label{T:liftingthm1}
  A projection $\mathbf{f}$ in $\mathcal{C}(C(X)\otimes B)$ represented by
  $(f_0,\cdots, f_n)$ under a suitable partition $\{ x_1,\dots, x_n \}$ of the interior of $X$ is liftable to a projection in $M(C(X)\otimes B)$  where $f_i(x)$'s are full and properly infinite projections for all $i$ and $x\in X$ if and only if there
  exist $l_0,\cdots,l_n$ in $K_0(B)$ satisfying the following conditions; suppose $k_i$'s are  the (generalized) ``essential codimensions'' of $f_i(x_i)$ and $f_{i-1}(x_i)$ for $1 \leq i \leq n$.
  \begin{equation}\label{E:eq1}
 l_i-l_{i-1}=-k_i \quad \mbox{for} \quad i>0 \quad \mbox{and}\quad l_0-l_n=-k_0 \quad \mbox{in the circle case,}
 \end{equation}
   if for some $x$ in $X_i$, $f_i(x)$ belongs to $B$, then
             \begin{equation}\label{E:eq2}
         l_i \geq - [f_i(x)]_0,
         \end{equation}
   if for some $x$ in $X_i$, $1-f_i(x)$ belongs to $B$, then
  \begin{equation}\label{E:eq3}
  l_i \leq [1-f_i(x)]_0,
  \end{equation}
   if either end point of $X_i$ is infinite, then
  \begin{equation}\label{E:eq4}
  l_i=0.
  \end{equation}
  \end{thm}

In addtion, we are going to show that two projections in the corona algebra are equivalent under suitable condtions on associated generalized essential codimensions.

  \section{The essential codimension}\label{S:codimension}

As we observe in Theorem \ref{T:liftingthm1}, the technical tool other than Lemma \ref{L:deform1} is a K-theoretic notion which is called the generalized essential codimension. In fact, this generalizes the notion of classical essential codimension of Brown, Douglas, Fillmore. In Section \ref{S:codimension}, we give a careful treatment of this notion using rudiments of Kasparov's $KK$-theory although it appeared in \cite{Lee} without showing the connection to the classical essential codimension.

Let $E$ be a (right) Hilbert $B$-module. We denote $\mathcal{L}(E,F)$ by the $C\sp{*}$-algebra of adjointable, bounded  operators from $E$ to $F$. The ideal of `compact' operators from $E$ to $F$ is denoted by $\mathcal{K}(E,F)$. When $E=F$, we write $\mathcal{L}(E)$ and $\mathcal{K}(E)$ instead of  $\mathcal{L}(E,E)$ and $\mathcal{K}(E,E)$. Throughout the paper, $A$ is a separable $C\sp{*}$-algebra, and all Hilbert modules are assumed to be countably generated over a separable $C\sp*$-algebra. We use the term representation for a $*$-homomorphism from $A$ to $\mathcal{L}(E)$. We let $H_B$ be the standard Hilbert module over $B$ which is $H\otimes B $ where $H$ is a separable infinite dimensional Hilbert space. We denote $M(B)$ by the multiplier algebra of $B$. It is well-known that $\mathcal{L}(H_B)=M(B\otimes K)$ and $\mathcal{K}(H_B)=B\otimes K$ where $K$ is the $C\sp{*}$-algebra of the compact operators on $H$ \cite{Kas80}.

Let us recall the definition of Kasparov group $\KK(A,B)$. We refer the reader to \cite{Kas81} for the general introduction of the subject. A $\KK$-cycle is a triple $(\phi_0,\phi_1,u)$, where $\phi_i:A \to \mathcal{L}(E_i)$ are representations and $u \in \mathcal{L}(E_0,E_1)$ satisfies that
\begin{itemize}
\item[(i)] $u\phi_0(a)-\phi_1(a)u \in \mathcal{K}(E_0,E_1)$,
\item[(ii)]$\phi_0(a)(u^*u-1)\in \mathcal{K}(E_0)$, $\phi_1(a)(uu^*-1)\in \mathcal{K}(E_1)$.
\end{itemize}
The set of all $KK$-cycles will be denoted by $\mathbb{E}(A,B)$.  A cycle is degenerate if
\[ u\phi_0(a)-\phi_1(a)u=0, \quad\phi_0(a)(u^*u-1)=0,\quad \phi_1(a)(uu^*-1)=0.\]
An operator homotopy through $KK$-cycles is a homotopy $(\phi_0,\phi_1,u_t)$, where the map $t \to u_t$ is norm continuous.
The equivalence relation $\underset{\text{oh}}{\sim}$ is generated by operator homotopy and addition of degenerate cycles up to unitary equivalence.
Then $\KK(A,B)$ is defined as the quotient of $\mathbb{E}(A,B)$ by $\underset{\text{oh}}{\sim}$. When we consider  non-trivially graded $C\sp{*}$-algebras, we define a triple $(E,\phi,F)$, where $\phi:A \to \mathcal{L}(E)$ is a graded representation, and $F\in \mathcal{L}(E)$ is of odd degree such that $F\phi(a)-\phi(a)F$, $(F^2-1)\phi(a)$, and $(F-F^*)\phi(a)$ are all in $\mathcal{K}(E)$ and call it a Kasparov $(A,B)$-module. Other definitions like degenerate cycle and operator homotopy are defined in  similar ways.

In the above, we introduced the Fredholm picture of $\KK$-group. There is an alternative way to describe the element of $KK$-group. The Cuntz picture is described by a pair of representations $\phi,\psi:A \to \mathcal{L}(H_B)=M(B\otimes K)$ such that $\phi(a)-\psi(a)\in \mathcal{K}(H_B)=B\otimes K$. Such a pair is called a Cuntz pair. They form a set denoted by $\mathbb{E}_h(A,B)$. A homotopy of Cuntz pairs consists of a Cuntz pair $(\Phi,\Psi):A \to M(C([0,1])\otimes (B\otimes K))$. The quotient of $\mathbb{E}_h(A,B)$ by homotopy equivalence  is a group $\KK_h(A,B)$ which is isomorphic to $\KK(A,B)$ via the mapping sending $[\phi,\psi]$ to $[\phi,\psi,1]$.

\begin{defn}[the generalized essential codimension]\label{D:BDF}
Given two projections $p,q \in M(B\otimes K)$ such that $p-q \in B\otimes K$, we consider representations $\phi,\psi $ from $\mathbb{C}$ to $M(B\otimes K)$ such that $\phi(1)=p,\psi(1)=q$. Then $(\phi,\psi)$ is a Cuntz pair so that
we define $[p:q]$ as the class $[\phi,\psi]\in \KK_h(\mathbb{C},B) \simeq K(B)$ and call the (generalized) essential codimension of $p$ and $q$.
\end{defn}
\begin{rem}
We recall that BDF's original definition of the essential codimension of $p$ and $q$ in $B(H)$ is given by the Fredholm index of $V^*W$ where $V$ and $W$ are  isometries such that $VV^*=q$ and $W^*W=p$ \cite{BDF}. This looks different with the above definition. But in the case $B=K$ or $\mb{C}$  $[p:q]$ is mapped to $[\phi,\psi,1]$ in $\KK(\mb{C},\mb{C})$. Then the map from $\KK(\mb{C},\mb{C})$ to $\mb{Z}$ sends $[\phi,\psi,1]$ to the Fredholm index of $qp$ viewing $qp$ as the operator from $pH$ to $qH$. Thus it equals to the Fredholm index of $V^*W$.
\end{rem}

The following demonstrate that the generalized essential codimension behaves like the original essential codimension (see \cite[Section 1]{B}).

\begin{lem}\label{L:properties}
$[\,:\,]$ has the following properties.
\begin{enumerate}
\item $[p_1:p_2]=[p_1]_0-[p_2]_0$ if either $p_1$ or $p_2$ belongs to $B$, where $[p_i]_0$ is the $K_0$-class of a projection $p_i$,
 \item $[p_1:p_2]=-[p_2:p_1],$
 \item $[p_1:p_3]=[p_1:p_2]+[p_2:p_3],\text{when sensible},$
 \item $[p_1+q_1:p_2+q_2]=[p_1:p_2]+[q_1:q_2], \, \text{when sensible.}$
\end{enumerate}
\end{lem}
\begin{proof}
Let $\phi_i$'s and $\psi_i$'s  be elements in $\Hom(\mb{C},M(B))$ such that $\phi_i(1)=p_i$ and $\psi_i(1)=q_i$ respectively.
 Without loss of generality we let B be a stable $C\sp{*}$-algebra and $\Theta_B: M_2(B) \to B$ be an inner isomorphism.
Then
\begin{itemize}
\item[(1)]In general, the isomorphism \[KK_h(A,B) \to KK(A,B)\] maps
$[\phi_1,\phi_2]$ to a cycle $[\phi_1,\phi_2,1]$. Note that, when $A=\mb{C}$,
\[
(\phi_1,\phi_2,1)=\left(H_B \oplus H_B, \left(\begin{array}{cc}
                                                                            \phi_1 &  \\
                                                                             & \phi_2 \\
                                                                          \end{array}
                                                                        \right), \left(
                                                                                   \begin{array}{cc}
                                                                                     0 & 1 \\
                                                                                     1 & 0 \\
                                                                                   \end{array}
                                                                                 \right)
 \right)\] is a compact perturbation of \[\left(H_B \oplus H_B, \left(\begin{array}{cc}
                                                                            p_1 &  \\
                                                                             & p_2 \\
                                                                          \end{array}
                                                                        \right), \left(
                                                                                   \begin{array}{cc}
                                                                                     0 & p_1p_2 \\
                                                                                     p_2p_1 & 0 \\
                                                                                   \end{array}
                                                                                 \right)
 \right). \] The latter is decomposed to $((1-p_1)(H_B) \oplus( 1-p_2)(H_B),0,0)\oplus \left(p_1(H_B) \oplus p_2(H_B), \left(\begin{array}{cc}
                                                                            p_1 &  \\
                                                                             & p_2 \\
                                                                          \end{array}
                                                                        \right), \left(
                                                                                   \begin{array}{cc}
                                                                                     0 & p_1p_2 \\
                                                                                     p_2p_1 & 0 \\
                                                                                   \end{array}
                                                                                 \right)
 \right)$

 so that
 $(\phi_1,\phi_2,1)$ is represented as \[\left(p_1(H_B) \oplus p_2(H_B), \left(\begin{array}{cc}
                                                                            p_1 &  \\
                                                                             & p_2 \\
                                                                          \end{array}
                                                                        \right), \left(
                                                                                   \begin{array}{cc}
                                                                                     0 & p_1p_2 \\
                                                                                     p_2p_1 & 0 \\
                                                                                   \end{array}
                                                                                 \right)
 \right).\] Now we view $p_2p_1$ as an essential unitary operator from $p_1(H_B)$ to $p_2(H_B)$, thus its generalized Fredholm index is given by $[p_1]_0-[p_2]_0$ (see \cite{Mi}).
 In fact, it is realized as the following diagram
\[
 \begin{CD}
  p_1(H_B)\oplus (1-p_1)(H_B)@>>> p_2(H_B)\oplus (1-p_2)(H_B)\\
  @V I\oplus U^* VV                @A I\oplus W AA \\
  p_1(H_B)\oplus H_B         @> p_2p_1\oplus I >> p_2(H_B)\oplus H_B
  \end{CD}
 \]
 where $U,W \in M(B)=\mc{L}(H_B)$ are isometries such that $UU^*=1-p_1, WW^*=1-p_2$.
\item[(2)] $\left(
                                                           \begin{array}{cc}
                                                             \cos (\frac{\pi}{2}t) & \sin(\frac{\pi}{2}t) \\
                                                             -\sin(\frac{\pi}{2}t) & \cos (\frac{pi}{2}t)\\
                                                           \end{array}
                                                         \right)$ defines a homotopy from $\left(
                                                           \begin{array}{cc}
                                                             \phi_2 & 0 \\
                                                             0 & \phi_1 \\
                                                           \end{array}
                                                         \right) $  to $\left(
                                                           \begin{array}{cc}
                                                             \phi_1 & 0 \\
                                                             0 & \phi_2 \\
                                                           \end{array}
                                                         \right)$. Thus
 \[\begin{split}
[\phi_1,\phi_2]+[\phi_2,\phi_1]&=\left[ \Theta_B\circ \left(
                                                           \begin{array}{cc}
                                                             \phi_1 & 0 \\
                                                             0 & \phi_2 \\
                                                           \end{array}
                                                         \right), \Theta_B\circ \left(
                                                           \begin{array}{cc}
                                                             \phi_2 & 0 \\
                                                             0 & \phi_1 \\
                                                           \end{array}
                                                         \right)
 \right]\\
 &=\left[ \Theta_B\circ \left(
                                                           \begin{array}{cc}
                                                             \phi_1 & 0 \\
                                                             0 & \phi_2 \\
                                                           \end{array}
                                                         \right), \Theta_B\circ \left(
                                                           \begin{array}{cc}
                                                             \phi_1 & 0 \\
                                                             0 & \phi_2 \\
                                                           \end{array}
                                                         \right)
 \right]\\
 &=0.
 \end{split}\]
\item[(3)] Similarly, \[\begin{split}
[\phi_1,\phi_2]+[\phi_2,\phi_3]&=\left[ \Theta_B\circ \left(
                                                           \begin{array}{cc}
                                                             \phi_1 & 0 \\
                                                             0 & \phi_2 \\
                                                           \end{array}
                                                         \right), \Theta_B\circ \left(
                                                           \begin{array}{cc}
                                                             \phi_2 & 0 \\
                                                             0 & \phi_3 \\
                                                           \end{array}
                                                         \right)
 \right]\\
 &=\left[ \Theta_B\circ \left(
                                                           \begin{array}{cc}
                                                             \phi_2 & 0 \\
                                                             0 & \phi_1 \\
                                                           \end{array}
                                                         \right), \Theta_B\circ \left(
                                                           \begin{array}{cc}
                                                             \phi_2 & 0 \\
                                                             0 & \phi_3 \\
                                                           \end{array}
                                                         \right)
 \right]\\
 &=[\phi_2,\phi_2]+[\phi_1,\phi_3]\\
 &=[\phi_1,\phi_3].
 \end{split}\]
 \item[(4)] Since $\phi_i$ and $\psi_i$ are orthogonal, i.e. $\phi_i \psi_i=0$,  $\phi_i+\psi_i$ is a homomorphism.  Note that $\left(
                                                                                                                                 \begin{array}{cc}
                                                                                                                                   \phi_i & 0 \\
                                                                                                                                   0 & \psi_i \\
                                                                                                                                 \end{array}
                                                                                                                               \right)$ is homotopic to  $\left(
                                                           \begin{array}{cc}
                                                             \phi_i+\psi_i & 0 \\
                                                              0 & 0 \\
                                                           \end{array}
                                                         \right)$ up to stability.
 \[
 \begin{split}
 [\phi_1,\phi_2]+[\psi_1,\psi_2]&=\left[ \Theta_B\circ \left(
                                                           \begin{array}{cc}
                                                             \phi_1 & 0 \\
                                                             0 & \psi_1 \\
                                                           \end{array}
                                                         \right), \Theta_B\circ \left(
                                                           \begin{array}{cc}
                                                             \phi_2 & 0 \\
                                                             0 & \psi_2 \\
                                                           \end{array}
                                                         \right)
 \right]\\
 &=\left[ \Theta_B\circ \left(
                                                           \begin{array}{cc}
                                                             \phi_1+\psi_1 & 0 \\
                                                             0 & 0 \\
                                                           \end{array}
                                                         \right), \Theta_B\circ \left(
                                                           \begin{array}{cc}
                                                             \phi_2+\psi_2 & 0 \\
                                                             0 & 0 \\
                                                           \end{array}
                                                         \right)
 \right]\\
 &=[\phi_1+\psi_1,\phi_2+\psi_2].
 \end{split}
 \]
 \end{itemize}
\end{proof}
 \begin{lem}\label{L:unitaryequi}
  Let $p$ and $q$ be projections in $M(B)$ such that $p-q \in B$. If there is a
  unitary $U \in 1+ B$ such that $UpU^*=q$, then $[p:q]=0$. In
  particular, if $\| p- q\|<1 $, then $[p:q]=0$.
 \end{lem}
 \begin{proof}
 Since $(\phi_1,\phi_2,1)$ is unitarily equivalent to $(\phi_1,\phi_1,1)$ which is a degenerate element, it follows that $[p:q]=0$ in $\KK(\mb{C},B)$

 If $\|p-q\| <1$, we can take $a=(1-q)(1-p)+qp \in 1+ B$. Since
 $aa^*=a^*a=1-(p-q)^2 \in 1+ B$,

   \[\|a^*a -1\|=\|p-q\|^2 < 1, \quad \|aa^*-1\|=\|p-q\|^2 < 1. \] Moreover, it
   follows that
     \begin{align*}
     ap &= qp =qa.
     \end{align*}
  Hence, $a$ is invertible element and $U=a(a^*a)^{-\frac{1}{2}} \in 1+B$
  is a unitary such that $UpU^*=q$.
 \end{proof}
\begin{prop}\label{P:restrictions}
 Let $p, q$ be projections in M(B) such that $p-q \in B$. Suppose that $(K_0(B),K_0(B)^{+})$ is an ordered group  where the positive cone $K_0(B)^{+}=\{[p]_0\mid p \in \mc{P}_{\infty}(B)\}$.
 \begin{itemize}
 \item [(1)] If $q \in B$, then $[p:q] \geq -[q]_0$,
 \item [(2)] If $1-q \in \widetilde{B}$, then $[p:q] \leq [1-q]_0 $.
 \end{itemize}
 \end{prop}
 \begin{proof}
 \begin{itemize}
 \item[(1)] Note that $[p:q]=[p]_0-[q]_0$ by Lemma \ref{L:properties}-(1). Hence $[p:q]\geq -[q]_0$.
 \item[(2)] Notice that $[1-p]_0-[1-q]_0\in K_0(B)$. By Lemma \ref{L:properties}-(4) $[p:q]=-[1-p:1-q]=[1-q]_0-[1-p]_0\leq [1-q]_0$.
 \end{itemize}
 \end{proof}
 \begin{lem}\label{L:homotopy}
 Suppose projections $p_t,  q_t \in M(B)$ are defined for each $t$ in a set of $\mathbb{R}$.  Then
 $[p_t:q_t]$ is constant if $t \to p_t - q_t$ is norm continuous  in $ B$.
 \end{lem}
 \begin{proof}
Straightforward.
 \end{proof}
\begin{rem}
 Originally, the proof of this lemma for $B=K$ and $M(B)=B(H)$ was nontrivial (see \cite[Corollary 2.6]{BL}). However, it is now built in the definition of $\KK_h$.
\end{rem}
The next theorem, which we do not claim its originality, was proved in \cite{Lee} exhibits the most important property of the essential codimension (see \cite[Theorem 2.7]{BL}).
 \begin{thm}\label{T:BDF}
Let $B$ be a non-unital ($\sigma$-unital) purely infinite simple $C\sp{*}$-algebra such that $M(B\otimes K)$ has real rank zero.
Suppose two projections $p$ and $q$ in $M(B\otimes K)=\mathcal{L}(H_B)$ such that $p-q \in B\otimes K$ and neither of them is in $B\otimes K$. If $[p:q] \in K_{0}(B)$ vanishes, then there is a unitary $u$ in $ 1 + B\otimes K$ such that $upu^*=q $.
\end{thm}
\begin{rem}\label{R:properequivalence}
\begin{itemize}
\item[(i)] As it was pointed out in \cite{DE} and \cite{Lee}, the crucial point of Theorem \ref{T:BDF} is that the implementing unitary $u$ has the form ``identity + compact''. This requirement to obtain the reasonable generalization of BDF's original statement has been very useful in $\KK$-theory (see \cite{DE} and \cite{Lee}).
\item[(ii)] Without restrictions on $C\sp*$-algebra  $B$, but on $p \in M(B\otimes K)$ being halving projection, any compact perturbation of $p$ is of the form $upu^*$ where a unitary $u$ is in $ 1 + B\otimes K$ \cite{Zh3}.
\end{itemize}    
\end{rem}

 \section{Deformation of a projection and its applications}\label{S:deformation}

 Let $B$ be a simple stable $C\sp*$-algebra such that the multiplier algebra $M(B)$ has real rank zero.
Let $X$ be $[0,1], [0,\infty)$, $(-\infty,\infty)$ or
$\mathbb{T}=[0,1]/\{0,1\}$. When $X$ is compact, let $I=C(X)\otimes
   B$ which is the $C\sp{*}$-algebra of (norm continuous) functions from $X$ to $B$. When $X$ is not compact,  let $I=C_0(X)\otimes B$ which is the $C\sp*$-algebra of continuous functions from $X$ to $B$ vanishing at infinity. Then $M(I)$ is given by $C_b(X, M(B)_s)$, which is the space of
   bounded functions from $X$ to $M(B)$, where $M(B)$ is given the strict
   topology. Let $\mathcal{C}(I)=M(I)/I$ be the corona algebra of $I$ and also
   let $\pi:M(I) \to \mathcal{C}(I)$ be the natural quotient map. Then an element
   $\mathbf{f}$ of the corona algebra can be represented as follows:  Consider a
 finite partition of $X$, or $X \smallsetminus \{0,1\}$ when $X=\mathbb{T}$ given by partition points $x_1 < x_2 < \cdots
 < x_n $ all of which are in the interior of $X$ and divide $X$ into
 $n+1$ (closed) subintervals $X_0,X_1,\cdots,X_{n}$. We can take $f_i \in
 C_b(X_i, M(B)_s)$ such that $f_i(x_i) -f_{i-1}(x_i)\in B$
 for $i=1,2,\cdots,n$ and $f_0(x_0)-f_n(x_0) \in B$  where $x_0=0=1$ if $X$ is $\mathbb{T}$. The following lemma and the statement after it were shown in \cite{Lee}.

 \begin{lem}
 The coset in $\mathcal{C}(I)$ represented by
 $(f_0,\cdots,f_n)$  consists of functions $f$ in $M(I)$ such that $f- f_i \in
 C(X_i)\otimes B$ for every $i$ and $f-f_i $ vanishes (in norm) at any
 infinite end point of $X_i$.
 \end{lem}
 Similarly $(f_0,\cdots,f_n)$ and $(g_0,\cdots,g_n)$
 define the same element of $\mathcal{C}(I)$ if and only if $f_i - g_i \in
 C(X_i)\otimes B$ for $i=0,\cdots,n$ if $X$ is compact.  $(f_0,\cdots,f_n)$ and $(g_0,\cdots,g_n)$
 define the same element of $\mathcal{C}(I)$ if and only if  $f_i - g_i \in
 C(X_i)\otimes B$ for $i=0,\cdots,n-1$, $f_n -g_n \in
 C_0([x_n,\infty))\otimes B$  if $X$ is $[0.\infty)$.  $(f_0,\cdots,f_n)$ and $(g_0,\cdots,g_n)$
 define the same element of $\mathcal{C}(I)$ if and only if
  $f_i - g_i \in
 C(X_i)\otimes B$ for $i=1,\cdots,n-1$, $f_n -g_n \in
 C_0([x_n,\infty))\otimes B$, $f_0-g_0 \in
 C_0((-\infty,x_1])\otimes B$ if $X=(-\infty,\infty)$.\\

A virtue of the above unnecessarily complicated description of an element in $\mc{C}(I)$ is the following theorem which says a projection in $\mc{C}(I)$ is locally liftable.
\begin{thm}\cite[Theorem 3.2]{Lee}\label{T:locallift}
Let $I$ be $C(X)\otimes B$ or $C_0(X)\otimes B$ where $B$ is a stable $C\sp{*}$-algebra such that $M(B)$  has real rank zero.
Then a projection $\mathbf{f}$ in $M(I)/I$ can be represented by $(f_0,f_1,\cdots, f_n)$ as above where $f_i$ is a projection valued function in
$C(X_i)\otimes M(B)_s$ for each $i$.
\end{thm}
\begin{rem}
The theorem says that any projection $\mathbf{f}$ in the corona algebra of $C(X)\otimes B$ for some $C\sp{*}$-algebras $B$ can be viewed as a ``locally trivial fiber bundle'' with the Hilbert modules as fibers in the sense of Dixmier and Duady \cite{DixDua}.
\end{rem}
The following theorem will be used as one of our technical tools.
\begin{thm}\cite[Theorem 3.3]{Lee}\label{T:lifting}
Let $I$ be $C(X)\otimes B$ where $B$ is a $\sigma$-unital, non-unital, purely infinite simple $C\sp{*}$-algebra such that $M(B)$  has real rank zero or $K_1(B)=0$ (See \cite{Zh}).
Let a projection $\mathbf{f}$ in $M(I)/I$ be represented by $(f_1,f_2,\cdots, f_n)$, where $f_i$ is a projection valued function in
$C(X_i)\otimes M(B)_s$ for each $i$, as in Theorem \ref{T:locallift}.
If $k_i=[f_i(x_i):f_{i-1}(x_i)]=0$ for all $i$, then the projection $\mathbf{f}$ in $M(I)/I$ lifts.
\end{thm}

 Now we want to observe the necessary conditions for a projection in $\mc{C}(I)$ to lift.
 If $\mathbf{f}$ is liftable to a projection $g$ in $M(I)$,  we
 can use the same partition of $X$ so that $(g_0,\cdots,g_n)$ and $(f_0,
 \cdots, f_n)$ define the same element $\mathbf{f}$ where $g_i$ is the
 restriction of $g$ on $X_i$. Then, for each $i$, $[g_i(x):f_i(x)]$ is
 defined for all $x$. By Lemma \ref{L:homotopy}  this function must be constant on $X_i$
 since $g_i -f_i$ is norm continuous. So we can let $l_i= [g_i(x):f_i(x)]$. \\
 Since $g_i(x_i)=g_{i-1}(x_i)$, we have
 $[g_i(x_i):f_i(x_i)]+[f_i(x_i):f_{i-1}(x_i)]=[g_{i-1}(x_i):f_{i-1}(x_i)]$ by
 Lemma \ref{L:properties}-(3). In other words,
 \begin{equation}\label{E:eq1}
 l_i-l_{i-1}=-k_i \quad \mbox{for} \quad i>0 \quad \mbox{and}\quad l_0-l_n=-k_0 \quad \mbox{in the circle case}.
 \end{equation}
 Moreover, if $(K_0(B),K_0(B)^{+})$ is an ordered group with the positive cone $K_0(B)^{+}=\{[p]_0\mid p \in \mc{P}_{\infty}(B)\}$, we apply Proposition \ref{P:restrictions} and Lemma \ref{L:unitaryequi} to
  projections $g_i(x)$ and $f_i(x)$, and we have the following restrictions on
  $l_i$ .
  \begin{itemize}
  \item[(i)] If for some $x$ in $X_i$, $f_i(x)$ belongs to $B$, then
             \begin{equation}\label{E:eq2}
         l_i \geq - [f_i(x)]_0,
         \end{equation}
  \item[(ii)] If  for some $x$ in $X_i$, $1-f_i(x)$ belongs to $B$, then
  \begin{equation}\label{E:eq3}
  l_i \leq [1-f_i(x)]_0,
  \end{equation}
  \item[(iii)] If either end point of $X_i$ is infinite, then
  \begin{equation}\label{E:eq4}
  l_i=0.
  \end{equation}
  \end{itemize}

 Are these necessary conditions sufficient?  Our strategy of showing the converse is to perturb the $f_i$'s so that $[f_i(x_i):f_{i-1(x_i}]$ vanishes for all i.
To do this we need to deform  a field of orthogonal projections by embedding a field a projections with arbitrary``rank'' information which is given by a K-theoretical term. This implies that a field of Hilbert modules defined by a field of projections should be ambient so that any field of submodules with arbitrary rank can be embedded. Unlike a Hilbert space there is no proper algebraic notion of rank of a multiplier projection whose image can be regarded as a Hilbert submodule of the standard Hilbert module $H_B$. However, at least we want to distinguish submodules in terms of finiteness and infiniteness of corresponding projections.
Recall that a projection  $p$ in a unital $C\sp*$-algebra $A$ is called a \emph{halving} projection if both $1-p$ and $p$ are Murray-von Neumann equivalent to the unit in $A$ and a projection $p$ in $A$ is called properly inifinite if there are mutually orthogonal projections $e$, $f$ in $A$ such that $e \le p$, $f \le p$, and $e \sim f \sim p$. Then it is easy to check that a projection in the multiplier algebra of a stable $C\sp*$-aglebra is Murray-von Neumann equivalent to $1$ if and only if it is full and properly infinite.  We denote by  $\mf{P}$ the set of all full and properly infinite projections in $M(B)$. We assume $M(B)$ has a halving projection and fix the halving projection $H$.

For the moment we allow $X$ to be any finite dimensional topological space.
Let $\mathbf{p}$ be a section which is a continuous map from $X$ to $M(B)$ with respect to the strict topology on $M(B)$. We denote its image on $x\in X$ by $p_x$. By the observation so far, it is natural to put the condition that $p_x$ is full and properly inifinte for every $x \in X$ to have an analogue of an infinite dimensinal Hilbert bundle or a continuous field of separable infinited dimensioanl Hilbert spaces.

We need the Michael selection theorem as the important step to get the main result Lemma \ref{L:deform}.
\begin{thm}(See \cite{Mich})\label{T:Selection}
Let $X$ be a paracompact finite dimensional topological space. Let $Y$ be a complete metric space and $S$ be a set-valued lower semicontinuous map from $X$ to closed subsets of $Y$, i.e. for each open subset $U$ of $Y$ the set $\{x \in X \mid S(x) \cap U \neq \emptyset \}$ is open. Let $\mathcal{R}$ be the range $\{S(x): x\in X\}$ of $S$. Then, if for some $m> \dim X$ we have
\begin{itemize}
\item[(i)] $S(x)$ is $m$-connected;
\item[(ii)] each $R \in \mathcal{R}$ has the property that every point $x \in R$ has an arbitrarily small neighborhood $V(x)$ such that $\pi_m(R^{'}\cap V(x))=\{e\}$ for every $R^{'} \in \mathcal{R}$;
\end{itemize}
then there exists a continuous map $s$ from $X$ to $Y$ such that $s(x)$ is in $S(x)$ for all $x \in X$.
\end{thm}
Here, $\pi_m$ is the $m^{th}$ homotopy group, defined by $\pi_m(Z)=[S^m, Z]$.

The following lemma was actually proved in \cite{KuNg} under the assumption that $p_x$ is halving for every $x$, but this assumption can be weaken keeping the same proof \cite{Ng}.
\begin{lem}\label{L:deform}
Let $B$ be a $\sigma_p$-unital, stable $C\sp*$-algebra and $\mathbf{p}$ be a section which is a map from a compact, finite dimensional topological space  $X$ to $M(B)$ with respect to the strict topology on $M(B)$. Suppose $p_x$  is a full, properly infinite projection in $M(B)$ for each $x$.
Then there exists a continuous map $u:X \to \mf{M}$ such that $u^*_x u_x=p_x$, $u_xu^*_x=H \in M(B)$, where $\mf{W}$ is the set of all partial isometries with the initial projection in $\mf{P}$ and range projection H.
\end{lem}
\begin{proof}
 We define a map $F_H: \mf{W} \to \mf{P}$ by $F_H(v)=v^*Hv$. Then it can be shown that $F_H^{-1}(p)$ is closed and contractible, and $F_H$ is an open mapping as in \cite{KuNg}. Let Y be norm closed ball in $M(B)$ of elements with the norm less than equal to 2. If we define a set-valued map $S:X \to 2^{Y}$ by $S(x)=F_H^{-1}(p_x)$, it can be shown that this map is lower semi-continuous using the openness of the map $F_H$ (see also \cite{KuNg}). Then we apply the Michael selection theorem to the map S so that there is a cross section $s:X \to Y$ such that $s(x)\in S(x)$. Then we define $u_x=Hs(x)$. Since $F_H(s(x))=p_x$, it follows that $u_xu^*_x=H$ and $u^*_xu_x=p_x$.
\end{proof}

 Recall that a closed submodule $E$ of the Hilbert module $F$ over $B$ is complementable if and only if there is a submodule $G$ orthogonal to $E$ such that $E\oplus G =F$. The Kasparove stabilization theorem says that a countably generated closed submodule of $H_B$ is complementable whence it is the image of a projection in $\mathcal{L}(H_B)$. Let $\mathfrak{F}=((F_t)_{\{t\in T\}},\Gamma)$ be a continuous field of Hilbert modules.  A continuous field of Hilbert modules $((E_x)_{\{x\in X\}},\Gamma')$ is said to be complementable to $\mathfrak{F}$ when $E_x$ is a complementable submodule of $F_x$ for each $x \in X$. Then the following is a geometrical interpretation of Lemma \ref{L:deform}.

 \begin{prop}
 A complementable subfield $((E_x)_{\{x\in X\}},\Gamma')$ of the constant module $((H_B)_{\{x\in X\}}, \Gamma)$, where $\Gamma$ consists of the (norm) continuous section from $X$ to $H_B$, is in one to one correspondence to a continuous projection-valued map $p: X \to \mathcal{L}(H_B)$, where the latter is equipped with the $*$-strong topology (In general, we say that $\{T_i\}$ in $\mathcal{L}(X)$ converges to $T$ $*$-strongly if and only if both $T_i(x) \to T(x)$ and $T_i^*(x) \to T^*(x)$ in $X$ for all $x\in X$).
 \end{prop}
 \begin{proof}
 Since $E_x$ is a complementable submodule of $H_B$, it is an image of a projection $p_x \in \mathcal{L}(H_B)$. This it defines a map $p:X \to \mathcal{L}(H_B)$. Note that in general when $\mathcal{H}=((H_x)_{x\in X}, \Gamma)$ is a continuous field of Hilbert modules and $((E_x)_{x\in X}, \Gamma^{'})$ is a complemented subfield of $\mathcal{H}$,  $x \mapsto p_x(\gamma(x))$ is continuous if and only if $ x \mapsto \|p_x(\gamma(x))\|$ is continuous for $\gamma \in \Gamma$.  In addition, $\Gamma^{'}=\{x \mapsto p_x(\gamma(x)) \,\, \text{ for $\gamma \in \Gamma$}\}$. So if $\mathcal{H}$ is a trivial field, the map $x \mapsto p_x(\xi)$ for $\xi \in H_B$  is in $\Gamma^{'}$, and thus continuous. This implies that the map $x \to p_x$ is strongly continuous.
 Conversely, suppose that we are given a strongly continuous map $x \to p_x \in \mathcal{L}(H_B)$. Let $E_x=p_x(H_B)$, and define a section $\gamma_{\xi}:x \to p_x(\xi)$ for each $\xi \in H_B$. Let $\Lambda=\{\gamma_{\xi} \in \prod_{x\in X} E_X \mid \xi \in H_B\}$ and $\Gamma^{''}=\{\gamma \in \prod_{x}E_x \mid \gamma \, \, \text{satisfies ($*$)}\}$.\\

 ($*$) For any $ x \in X$ and $\epsilon >0$, there exists $\gamma^{'} \in \bar{\Lambda}$ such that $\| \gamma(x)- \gamma^{'}(x)\| \le \epsilon$ in a neighborhood of $x$.\\

 Then we can check that $((E_x)_{x\in X},\Gamma^{''})$ is a complemented subfield of a trivial field.

 \end{proof}
Thus we have an analogue of well-known Dixmier's triviality theorem on continuous fields of Hilbert spaces in the Hilbert module setting. We note that the strict topology on $M(B\otimes K) \simeq \mathcal{L}(H_B)$ coincides with the $*$-strong toplogy on bounded sets (see \cite[Proposition C.7]{RaeWi}).
\begin{cor}\label{C:triviality}
Let $B$ be a stable $C\sp*$-algebra and X a finite dimensional compact Hausdorff space. Then a complementable subfield of Hilbert modules associated with a projection-valued map $p:X \to M(B)_s$ is isomorphic to a trivial field provided that each $p_x$ is a full, properly infinite projection in $M(B)$.
\end{cor}
\begin{proof}
 Lemma \ref{L:deform} says that $p \in M(C(X)\otimes B)$ is globally full, and properly infinite if $p_x$ is full and properly infinite for each $x$ in $X$. Thus $p$ is Murray-von Neumann equivalent to $1_{M(C(X)\otimes B)}$.
\end{proof}
\begin{lem}
Under the same hypothesis on $B$ and $X$ as in Lemma \ref{C:triviality} let $\mathbf{p}$ be a section which is a map from $X$ to $M(B)$ with respect to the strict topology on $M(B)$.
Suppose that there exist a continuous map $u:X \to \mf{P}$ such that $u^*_x u_x=p_x$, $u_xu^*_x=H \in M(B)$. In addition, given an $\alpha \in K_0(B)$ we suppose that there exist a projection $q \in HBH \subset B$ such that $[q]=\alpha \in K_0(B)$. Then there exists a norm continuous section $\mathbf{r}$ from $X$ to $B$ such that $r_x \leq p_x$ such that $[r]_{K_0(B)}=\alpha$.
\end{lem}
\begin{proof}
Since $q \in B$, $x \to qu_x \in B $ is norm continuous so that $x \to r_x=(qu_x)^*qu_x=u_x^*qu_x$ is norm continuous. Note that $(qu_x)^*qu_x=qu_xu_x^*q=qHq=q$. Thus $[r]_{K_0(B)}=[r_x]=[q]=\alpha$.
\end{proof}
 In summary, we state what we need as a final form .
\begin{lem}\label{L:subprojection}
Let $\mathbf{p}$ be a section from $X$ to $M(B)$ with respect to the strict topology on $M(B)$  where  $B$ is a $\sigma_p$-unital, stable $C\sp*$-algebra of real rank zero such that $M(B)$ contains a halving full projection. In addition, assume that $p_x$ is a full, properly infinite projection for each $x$. Then for any $\alpha \in K_0(B)$ there exists a norm continuous section $\mathbf{r}$ from $X$ to $B$ such that  $r_x \leq p_x$ such that $[r]_{K_0(B)}=\alpha$.
\end{lem}
\begin{proof}
If we denote a halving (strictly) full projection by $H$, $HBH$ is a full hereditary subalgebra of $B$ so that it is stably isomorphic to $B$ by \cite[Corollary 2.6]{Br}. Hence $K_0(HBH)=K_0(B)$. Since $B$ is a $C\sp*$-algebra of real rank zero, so is $HBH$ by \cite[Corollary 2.8]{BP}. This it satisfies the strong $K_0$-surjectivity, i.e. there exists a projection $q$ in $HBH$ such that $[q]_0=\alpha$. Then the conclusion follows from Lemma \ref{L:deform} and Lemma \ref{L:subprojection}.
\end{proof}

  Now we restrict ourselves to the case $I=C(X)\otimes B$ where $X$ is $[0,1]$, $[0,\infty)$, $(-\infty,\infty)$, or $[0,1]/\{0,1\}$, and $B$ is $\sigma_p$-unital, purely infinite simple $C\sp*$-algebra such that $M(B)$ has real rank zero and has a full halving projection $H$. From now on we assume that $K_0(B)$ is an odered abelian group and drop $0$ in the expression of an element in $K_0(B)$. Note that $B$ is a stable $C\sp*$-algebra by Zhang's diachotomy \cite{Zh}. Also, it satisfies the strong $K_0$-surjectivity \cite{Lin96}. Thus we can apply Lemma \ref{L:subprojection} to (one dimensional) closed intervals $X_i$'s, which come from a partition of $X$ associated with a local representation of a projection in the corona algebr of $I$.
\begin{thm}\label{T:liftingthm}
  A projection $\mathbf{f}$ in $\mathcal{C}(I)$ represented by
  $(f_0,\cdots, f_n)$ is liftable to a projection in $M(I)$  where $f_i(x)$'s  are halving projections for all $i$ and $x\in X$ if and only if there
  exist $l_0,\cdots,l_n$ satisfying above conditions
  (\ref{E:eq1}), (\ref{E:eq2}), (\ref{E:eq3}), (\ref{E:eq4}).
  \end{thm}
 \begin{proof}
  Given $l_i$'s satisfying (\ref{E:eq1}), (\ref{E:eq2}), (\ref{E:eq3}),
  (\ref{E:eq4}),
  we will show there exist $g_0, \cdots, g_n$
  such that $[g_i(x_i):g_{i-1}(x_i)]=0$ for $i>0$ and $[g_0(x_0):g_n(x_0)]=0$ in the circle case.

  First observe that if we have $g_{i}$'s such that $l_i=[g_i(x_i):f_i(x_i)]$,
  we have $[g_i(x_i):g_{i-1}(x_i)]=0 $ by (\ref{E:eq1}). Thus it is enough to
  show that there exist $g_0, \cdots, g_n$
  such that $[g_i(x_i):f_{i}(x_i)]=l_i$.
  \begin{itemize}
  \item[$l_i=0$]: Take $g_{i}=f_{i}$.
  \item[$l_i>0$]: By Lemma \ref{L:subprojection} the continuous field determined by $1-f_{i}$ has a
  trivial subfield which is given by a projection valued function
  $q \leq 1-f_i$ such that $[q(x)]_0=l_i$.  So we take $g_i=f_i+q$.
  \item[$l_i<0$]: Similarly, the continuous field determined by $f_i$ has a
   trivial subfield which is given by a projection valued function $q' \leq
  f_i$ such that $[q'(x)]_0=-l_i$. So we take $g_i=f_i-q'$.
 \end{itemize}
  Then the conclusion follows from Theorem \ref{T:lifting}.
  \end{proof}

Then we want to investigate some equivalence relations for projections using above arguments.  As before, let $\mathbf{p}$ and $\mathbf{q}$ be two projections in $\mathcal{C}(I)$.
\begin{lem}\label{L:onetoone}
     Let $(p_0, \cdots,
    p_n)$ and $(q_0, \cdots, q_n)$ be local liftings of $\mathbf{p}$
    and $\mathbf{q}$ such that $q_i(x)$ is a halving projection
    for each $x$ in $X_i$.

    If $ \sum_{i=1}^n[p_i(x_i):p_{i-1}(x_i)]=\sum_{i=1}^n
    [q_i(x_i):q_{i-1}(x_i)]$, or  $\sum_{i=1}^n[p_i(x_i):p_{i-1}(x_i)]+ [p_0(x_0):p_n(x_0)]=\sum_{i=1}^n
    [q_i(x_i):q_{i-1}(x_i)]+[q_0(x_0):q_n(x_0)]$ in the circle case, then we can find a perturbation  $(q_0', \cdots,
    q_n')$ of $\mathbf{q}$ such that $[p_i(x_i):p_{i-1}(x_i)]=
    [q_i'(x_i):q_{i-1}'(x_i)]$ for $i=1,\dots, n $ or $[p_i(x_i):p_{i-1}(x_i)]=
    [q_i'(x_i):q_{i-1}'(x_i)]$for $i=1,\dots, n+1$ modulo $n+1$.
     \end{lem}
     \begin{proof}
      Let $[p_i(x_i):p_{i-1}(x_i)]=k_i, [q_i(x_i):q_{i-1}(x_i)]=l_i$.
      If $d_i=k_i-l_i$, note that
      $$ \sum[p_i(x_i):p_{i-1}(x_i)]=\sum
    [q_i(x_i):q_{i-1}(x_i)] \quad \text{if and only if} \quad \sum d_i =0. $$
      Let $q_0'=q_0$. Suppose that we have constructed $q_0',
\cdots, q_{i}'$ such that $[p_j(x_j):p_{j-1}(x_j)]=
    [q_j'(x_j):q_{j-1}'(x_j)]$ for $j=1,\cdots,i$ and
    $[q_{i+1}(x_{i+1}):q_{i}'(x_{i+1})]=l_{i+1}-\sum_{k=1}^{i}
    d_k$. \\
    Let $\mathbf{r}$ be a projection valued (norm continuous) function on $X_{i+1}$ such that
    $r \leq 1-q_{i+1}$ and $[q]_{K_0(B)}=d_{i+1}+\sum_{k=1}^{i}
    d_k$. Then take $q_{i+1}'=q+q_{i+1}$. Then
    \begin{align*}
    [q_{i+1}'(x_{i+1}):q_{i}'(x_{i+1})]&=[q_{i+1}(x_{i+1}):q_{i}'(x_{i+1})]+[q(x_{i+1}):0] \\
                                              &=l_{i+1}-\sum_{k=1}^{i} d_k + d_{i+1}+\sum_{k=1}^{i} d_k\\
                                              &=l_{i+1}+k_{i+1}-l_{i+1}\\
                                              &=k_{i+1}
    \end{align*}
    \begin{align*}
    [q_{i+2}(x_{i+2}):q_{i+1}'(x_{i+2})]&=[q_{i+2}(x_{i+2}):q_{i+1}(x_{i+2})]+[0:q(x_{i+2})] \\
                                              &=l_{i+2} -( d_{i+1}+\sum_{k=1}^{i} d_k)\\
                                              &=l_{i+2}-\sum_{k=1}^{i+1} d_k
   \end{align*}
    By induction, we can get $q_{0}', \cdots, q_{n-1}'$ such
    that $[p_j(x_j):p_{j-1}(x_j)]=
    [q_j'(x_j):q_{j-1}'(x_j)]$ for $j=1,\cdots,n-1$
   as we want.
    Finally, since we also have $[q_{n}(x_{n}):q_{n-1}'(x_{n})]=l_{n}-\sum_{k=1}^{n-1}
    d_k=l_{n}+d_{n}=k_{n}$ from $\sum_{k=1}^{n-1}
    d_k+d_{n}=0$, we take $q_{n}'=q_{n}$.

    In the circle case, we perturb $q_n$ to $q'_n$ such that $[q_n'(x_n):q'_{n-1}(x_n)]=k_n$ and $[q_0(x_0):q'(x_0)]=l_0- \sum_{k=1}^n d^k=l_0+d_0=k_0$.
   \end{proof}
    Next is an analogous result that is more symmetrical.
\begin{lem}\label{L:samerank}
    Let $(p_0, \cdots,
    p_n)$ and $(q_0, \cdots, q_n)$ be local liftings of $\mathbf{p}$
    and $\mathbf{q}$ such that $p_i(x)$ and $q_i(x)$ are full, properly infinite projections
    for each $x$ in $X_i$.\\
    If $ \sum[p_i(x_i):p_{i-1}(x_i)]=\sum
    [q_i(x_i):q_{i-1}(x_i)]$, or  $[p_i(x_i):p_{i-1}(x_i)]=
    [q_i'(x_i):q_{i-1}'(x_i)]$for $i=1,\dots, n+1$ modulo $n+1$, then we can find  perturbations  $(q_0', \cdots,
    q_n')$ of $\mathbf{q}$ and $(p_0', \cdots,
    p_n')$ of $\mathbf{p}$ such that $[p_i'(x_i):p_{i-1}'(x_i)]=
    [q_i'(x_i):q_{i-1}'(x_i)]$ for all $i$.
   \end{lem}
   \begin{proof}
    The proof proceeds as above with one exception: If $d_{i+1}+\sum_{k=1}^{i} d_k \geq
    0$, we make $p_{i+1}' \leq p_i$ rather than making $q_{i+1}'\geq
    q_{i}$.
  \end{proof}
 An operator on $H_B$ is called a Fredholom operator when it is invertible modulo $\mathcal{K}(H_B)$ the ideal of compact operators. In fact, a generalized Atkinson theorem says that an opertor $F$ for which there exists a compact $K\in \mathcal{K}(H_B)$ such that $\Ker (F+K)$ and $\Ker (F+K)\sp*$ are finitely generated and $\Im F+K$ is closed is a Fredholm operator and vice virsa by \cite {Mi}. Thus we can define an index of a Fredholm operator in $K_0(B)$ as the diffence of two classes of finitely generated modules. Let us denote its index by $\Ind$. For more details, we refer the reader to \cite{We,Mi}.
\begin{prop}\label{P:equivalence}
Suppose $\mathbf{p}$ and $\mathbf{q}$ are given by projection valued functions $(p_0,p_1,\dots,p_n)$ and $(q_0,q_1,\dots,q_n)$, where both $p_i(x)$ and $q_i(x)$ are full and properly infinite projections for each $x$ in $ X_i$.
If $\sum_i k_i=\sum_i l_i$, then $\mathbf{p} \sim \mathbf{q}$.
\end{prop}
\begin{proof}
 By Lemma \ref{L:samerank} and the assumption we may arrange that $k_i=l_i$ for each $i$.
    Since $p_i(x)$ and $q_i(x)$ are full, properly infinite projections for each $x$ in  $X_i$, there is a (double) strongly
    continuous function $u_i$ on each $X_i$ such that ${u_i}^{\ast}u_i=p_i,
    u_i{u_i}^{\ast}=q_i$  by Corollary \ref{C:triviality}.
    Note that $u_{i-1}(x)$ is a unitary from $p_{i-1}(x)H$ onto $q_{i-1}(x)H$ so that $\Ind(u_{i-1}(x_i))=0$.  Then $k_i=l_i$ implies that
    $$\Ind(q_i(x_i)u_{i-1}(x_i)p_i(x_i))=
   -l_i+\Ind(u_{i-1}(x_i))+k_i=0,$$
    where the first index is for maps from $p_i(x_i)H_B$ to $q_i(x_i)H_B$, and, for example, the index of $p_{i-1}(x_i)p_i(x_i)$ as a map from $p_i(x_i)H_B$ to $p_{i-1}(x_i)H_B$ is $k_i$.
    Also
    $$q_i(x_i)u_{i-1}(x_i)p_i(x_i)-u_{i-1}(x_i) \in B.$$  There is  a compact perturbation $v_i$ of $q_i(x_i)u_{i-1}(x_i)p_i(x_i)$ such that $v_i^*v_i=p_i(x_i)$, $v_iv_i^*=q_{i}(x_i)$, and $v_i-u_{i-1}(x_i)\in B$.

      By the triviality of the continuous
    field of Hilbert modules determined by $p_{i}$ and the path connectedness of the
    unitary group of $M(B)$ \cite{Mi}, there is a path $\{v(t) : t\in [x_i,x]\}$
    such that $ v(t)^*v(t)=v(t)v(t)^*=p_i(t)$, $v(x_i)={u_i(x_i)}^{*}v_i$,  and $v(x)=p_{i}(x)$  for some $x \in
    X_i$. Then we  let $w_i=u_iv$ on $[x_i,x]$ so that
    \begin{align*}
    &w_i(x_i)-u_{i-1}(x_i) = v_i
     - u_{i-1}(x_i) \in
    B, \\
    &w_{i}^*w_{i} =v^*{u_i}^*u_iv=v^*p_iv=p_i, \\
    & w_{i}w_{i}^* =u_ivv^*{u_i}^*=u_ip_i{u_i}^*=q_i.
     \end{align*}
     Finally, we define
     \[
     u_i'=
             \begin{cases}
         w_i, \quad \text{on} \quad[x_i,x], \\
     u_i, \quad \text{on} \quad[x,x_{i+1}].
             \end{cases}
     \]
     In the $(-\infty,\infty)$-case we do the above for $i=1,\dots,n$ and let $u'_0=u_0$. In the circle case we do it for $i=0,\dots,n$.
     \end{proof}
\begin{cor}\label{C:uequivalence}
Suppose $\mathbf{p}$ and $\mathbf{q}$ are given by projection valued functions $(p_0,p_1,\dots,p_n)$ and $(q_0,q_1,\dots,q_n)$, where both $p_i(x)$ and $q_i(x)$ are halving projections for each $x$ in $ X_i$.
If $\sum_i k_i=\sum_i l_i$, then $\mathbf{p} \sim_{u} \mathbf{q}$.
\end{cor}

\begin{cor}
Suppose $\mathbf{p}$ and $\mathbf{q}$ are given by projection valued functions $(p_0,p_1,\dots,p_n)$ and $(q_0,q_1,\dots,q_n)$.
If $\sum_i k_i=\sum_i l_i$, then $[\mathbf{p}] \sim [\mathbf{q}]$ in $K_0$.
\end{cor}
\begin{proof}
Since $p_i(x)$ and $q_i(x)$ are halving, we apply Proposition \ref{P:equivalence} to $\mathbf{1-p}$ and $\mathbf{1-q}$, and obtain that $\mathbf{1-p} \sim \mathbf{1-q}$. It follows that $\mathbf{p} \sim_u \mathbf{q}$.
\end{proof}
\begin{proof}
 We replace $\mathbf{p}$ with $\mathbf{p}\oplus \mathbf{1}\oplus \mathbf{0}$ and $\mathbf{q}$ with $\mathbf{q}\oplus \mathbf{1}\oplus \mathbf{0}$, still being equal in $K_0$. Also, note that $k_i$'s and $l_i$'s are not changed. Then by Kasparov's absorption theorem $p_x(H_B)\oplus H_B \simeq H_B$ and so is $q_x(H_B)$. Thus $p_x \oplus 1$ and $q_x\oplus 1$ are Murray-von Neumann equivalent $1$. This implies that $p_x\oplus 1$ and $q_x \oplus 1$ are full, properly infinite projections in $M(B)$. The conclusion follows from Proposition \ref{P:equivalence}.
\end{proof}

\section{Acknowledgements}
The author wishes to thank S. Zhang for pointing out his previous result \ref{R:properequivalence}-(ii) during a conference. He also wishes to thank P.W. Ng for confirming Lemma \ref{L:deform}.

\end{document}